\documentclass[11pt,a4paper,reqno]{amsart}

 \usepackage{geometry,dirtytalk,cases}
 \usepackage{xcolor}
\usepackage{empheq,hyperref}
\usepackage[giveninits=true]{biblatex}

\theoremstyle{plain}
\begingroup
\newtheorem{theorem}{Theorem}[section]
\newtheorem{lemma}[theorem]{Lemma}
\newtheorem{proposition}[theorem]{Proposition}
\newtheorem{corollary}[theorem]{Corollary}
\endgroup

\hypersetup{			
			colorlinks=true,
			linkcolor=blue,
                        linktoc=page,
			anchorcolor=black,
			citecolor=blue,
			urlcolor=blue,
}

\theoremstyle{definition}
\begingroup
\newtheorem{definition}[theorem]{Definition}
\newtheorem{remark}[theorem]{Remark}
\endgroup

\numberwithin{equation}{section}

\bibliography{Sources}

\newcommand{\be}{\begin{equation}}
\newcommand{\ee}{\end{equation}}
\newcommand{\e}{\varepsilon}
\newcommand{\M}{{\mathcal M}}
\newcommand{\R}{{\mathbb R}}
\newcommand{\Sn}{{\mathbb S}}
\newcommand{\N}{{\mathbb N}}
\newcommand{\Z}{{\mathbb Z}}
\newcommand{\Rk}{{\R}^k}
\newcommand{\Rd}{{\R}^d}
\newcommand{\Rkd}{\R^{k\times d}}
\newcommand{\setmeno}{\!\setminus\!}
\newcommand{\hd}{{\mathcal H}^{d-1}}
\newcommand{\Ld}{{\mathcal{L}}^d}
\newcommand{\GBVsvector}{GBV_\star(A;\Rk)}
\newcommand{\GBVsscalar}{GBV_\star(A)}
\newcommand{\BVvect}{BV(A;\Rk)}
\newcommand{\BVsc}{BV(A)}
\newcommand{\A}{{\mathcal A}}
\newcommand{\B}{{\mathcal B}}

\newcommand{\trunc}[2]{#1^{(#2)}}
\newcommand{\jump}[1]{J_{#1}}
\newcommand{\mres}{\mathbin{\vrule height 1.6ex depth 0pt width
0.13ex\vrule height 0.13ex depth 0pt width 1.3ex}}

\title[Generalised vector-valued functions of Bounded Variation]{A new space of  Generalised vector-valued functions of Bounded Variation}
\begin{document}

\author[Davide Donati]{Davide Donati}
\address[Davide Donati]{SISSA, Via Bonomea 265, 34136 Trieste,
Italy}
\email[Davide Donati]{ddonati@sissa.it}

\begin{abstract}
    In \cite{DalToa22}, the authors introduced the space of  scalar-valued functions $\GBVsscalar$ to minimise a class of functionals whose study is motivated by fracture mechanics. In this paper, we extend the definition of $\GBVsscalar$ to the vectorial case, introducing the space $\GBVsvector$. We study the main properties of $\GBVsvector$ and  prove a lower semicontinuity result useful for minimisation purposes. With the Direct Method in mind, we adapt the arguments of \cite{DalToa22} to show that minimising sequences in $\GBVsvector$ can be modified to obtain a minimising sequence converging $\Ld$-a.e in $A$.
\end{abstract}
\maketitle
\medskip
 

\section{Introduction}
In his seminal paper \cite{Griffith}, Griffith introduced the idea that the stability or growth of the crack path in a material  is  determined by the competition of two energies acting at two different scales: the bulk energy spent to elastically deform the material and the surface energy spent to widen the pre-existing crack. 
In \cite{francfort1998revisiting}, Francfort and Marigo revisited Griffith's theory and proposed a model of crack growth based on an energy minimisation principle, renewing the interest in the study of energy functionals of the type
\begin{equation}\label{eq:Prototype Functional introduction}
\int_{A} f(x,\nabla u)\, dx+\int_{\jump{u}} g(x, [u],\nu_u)\,d\hd,
\end{equation}
where $A\subset \Rd$ is an open bounded set, the displacement $u\colon A\to\Rk$ can vary among different spaces of functions with bounded variation, $[u]=u^+-u^-$  represents the crack opening along the crack surface $\jump{u}$, $\nu_u$ is the unit normal at $\jump{u}$, and $\nabla u$ is the approximate gradient of $u$. 

The most common examples of such functionals  satisfy $g\equiv$ constant (see for instance \cite{Bourdin-Francfort-Marigo}), a choice which for lower semicontinuity reasons (see \cite{Ambrosio1989}) forces $f$ to be superlinear and requires to consider $SBV(A;\Rk)$ as underlying space. Many relevant functionals fall in this class. For instance, setting $f(x,\xi)=f(\xi)=|\xi|^p$, for $p>1$ one obtains the widely studied Mumford-Shah functional.

However, the hypothesis that $g$ is independent of the crack opening $[u]$ is in general not fully satisfactory in describing the fracture of materials, as cohesive forces are observed to be in action along the crack surface (see \cite{Barenblatt1962THEMT,Dugdale1960YieldingOS}).   This last observation suggests the introduction of cohesive models, in which integrands $g$ that are linear around the origin are considered,  the most simple of this kind of integrands being  $g(x,\zeta,\nu)=|\zeta|\land 1$, where for $a,b\in\R$, $a\land b$ is the minimum between $a$ and $b$. If one considers functions $g$ satisfying growth conditions of type 
\begin{equation}\label{eq:Dugdale bound introduction}
c(|\zeta|\land 1)\leq g(x,    \zeta,\nu)\leq C( |\zeta|\land 1),
\end{equation}
for suitable constants $0<c\leq C$, lower semicontinuity reasons (see \cite[Theorem 3.1]{Braides1995}) force $f$ to have linear growth and to consider an additional term in functionals \eqref{eq:Prototype Functional introduction}. The object of study  then becomes
\begin{equation}\label{eq:Functional Completo Intro}
    F(u)=\int_{A} f(x,\nabla u)\, dx+\int_A f^\infty\Big(x,\frac{dD^cu}{d|D^cu|}\Big)\, d|D^cu|+\int_{\jump{u}} g(x, [u],\nu_u)\,d\hd,
\end{equation}
where $f^\infty$ is the recession function of $f$, $D^cu$ is the Cantor part of the displacement $u$,  which one may interpret as a term carrying information about the formation of microfractures at a diffuse scale, and $dD^cu/d|D^cu|$ is the Radon-Nikod\'ym derivative of $D^cu$ with respect to $|D^cu|$.

Functionals of type \eqref{eq:Functional Completo Intro} with $g$ satisfying \eqref{eq:Dugdale bound introduction} and approximations of such functionals
have been extensively studied in recent years (see, for instance, \cite{Alicandro-Focardi,Bonacini-Conti-Iurlano,bonacini2023convergence,Conti-Focardi-Iurlano,conti2022phasefield, DalMaso-Orlando-Toader,DalToa23,DalToa22,DalToa23b}).

The choice of the function space where to study the functional $F$ is a delicate matter. Since by \eqref{eq:Dugdale bound introduction} $g$ does not control the full amplitude of  the crack openings $|[u]|$, one is tempted to consider as ambient space Ambrosio's $GBV(A)^k$ (see \cite[Section 1]{Ambrosio1990ExistenceTF} and \cite[Chapter 4]{AmbFuscPall}), where all the quantities appearing in \eqref{eq:Functional Completo Intro} are well-defined, including $D^cu$ (see for instance \cite[Lemma 2.10]{Alicandro-Focardi}). However, under the growth conditions \eqref{eq:Dugdale bound introduction}, the sublevels of $F$ are not precompact in $GBV(A)^k$,  and compactness is satisfied only by adding to $F$ lower order terms.

In the scalar case, a possible choice  is the space $GBV_\star(A)$, introduced by Dal Maso and Toader  in \cite{DalToa22}. In \cite{DalToa23,DalToa22} the authors are able to adapt the arguments of  \cite{friedrich2019compactness} to prove  that for certain minimum problems with Dirichlet boundary conditions, minimising sequences $(u_n)_n\subset \GBVsscalar$ admit a modification $(y_n)_n\subset \GBVsscalar$ which is still minimising and which converges $\Ld$-a.e. to $y\in \GBVsscalar$. This compactness property is achieved without any additional control on lower order terms.

  Adopting this last point of view, the aim of the present work  is twofold. First, we address the problem of extending the definition of $\GBVsscalar$  to the vectorial case and describing the main features of   $GBV_\star(A;\Rk)$.
 Then, we show that the compactness results of \cite{DalToa23,DalToa22} can be extended to $GBV_\star(A;\Rk)$ and give some sufficient condition for lower semicontinuity of integral functionals whose domain is $GBV_\star(A;\Rk)$. 
 These results will be used in \cite{DalDon} to deal with the  homogenisation of  functionals of cohesive type in the vectorial setting . 

The structure of the paper is as follows. In Section 3 we introduce the space $\GBVsvector$ and we present some of its main properties.  We later study the relation between a  function $u\in GBV_\star(A;\Rk)$ and $\phi\circ u$, for $\phi$ a Lipschitz function with compact support. We show in Proposition \ref{prop:characterisation} that $u\in GBV_\star(A;\Rk)$ if and only if for every $\phi$ as above, the function $\phi\circ u $ belongs to $BV(A;\Rk)$ and some relevant quantities associated with the derivatives of $\phi\circ u$ are  controlled from above by a constant $M$, independent of $\phi$, and $\textup{Lip}(\phi)$, the Lipschitz constant of $\phi$. We then analyse in Proposition \ref{prop:properties of CantorPart} the relation between $D^c(\phi\circ u)$ and $D^cu $ and use this result in Corollary \ref{prop:Alberti in GBVs} to show that $D^cu$ satisfies the Rank-One property.

In Section 4 we investigate the lower semicontinuity of a class of functionals with respect to the convergence in measure. In Theorem \ref{thm:lowersemicontinuity}, we prove the lower semicontinuity of certain functionals depending on $\nabla u$ and whose domain is the space $\GBVsvector$. 

In the rest of the section we study the compactness properties of the space $GBV_\star(A;\Rk)$. We first show that the following compactness result (Theorem \ref{thm:Compactness naive}) holds: if $(u_n)_n\subset\GBVsvector$ is a minimising sequence for $F$ given by \eqref{eq:Functional Completo Intro} with $g$ satisfying \eqref{eq:Dugdale bound introduction} and $f$ with linear growth, and if 
\begin{equation}\label{eq:intro compactness}
    \sup_{n\in\mathbb{N}}\int_Ah(|u_n|)\, dx<+\infty,
\end{equation}
for some continuous increasing function $h $ with $h(t)\to+\infty $ for $t\to+\infty$, then there exists a subsequence of $(u_n)_n$ converging $\Ld$-a.e. to some 
function  $u\in\GBVsvector$.
 
 We conclude by showing (Theorem \ref{thm:Compactness}) that using the arguments of \cite[Theorem 5.5] {DalToa22} and of \cite[Theorem 7.13]{DalToa23}, for every $\e_n\to 0^+$ every sequence of functions $(u_n)_n\subset GBV_\star(A;\Rk)$ satisfying some fixed common Dirichlet boundary conditions and such that $\sup_{n\in\N}F(u_n)< +\infty$  admits a modification $(y_n)_n\subset \GBVsvector$, with $y_n=u_n$ on $\partial A$, such that $F(y_n)\leq F(u_n)+\e_n$  and satisfying  \eqref{eq:intro compactness} with $u_n$ replaced by $y_n$.

\section{Notation and preliminaries}
We fix some notation that will be used throughout the paper.

\begin{enumerate}
    \item [(a)]  Given $n\in\N$, the symbol $\cdot$ denotes the scalar product of $\R^n$ and $|\,\,|$ the Euclidean norm of $\R^n$. If $a\in\R^n$, $a_i$ is the $i$-th component of $a$. The unit sphere of $\R^n$ is denoted by  $\Sn^{n-1}:=\{x\in\R^n\colon |x|=1\}$. Given $x\in\R^n$ and $\rho>0$, the open ball of center $x$ and radius $\rho$ is denoted by $B_\rho(x)$.
    \item [(b)] 
    Given $k,d\in\N$, we identify vectors in $\R^{k\times d}$ with $k\times d$ matrices. Given a matrix $\xi\in \R^{k\times d}$ and a vector $x\in\R^d$, the vector $\xi x$ is defined via the usual matrix by vector product. Given $\xi=(\xi_{ij})\in\R^{k\times d}$ the Frobenius norm of $\xi$ is  defined by
    \begin{equation}\nonumber
       |\xi|= \Big(\sum_{i=1}^k\sum_{j=1}^d\xi_{ij}^2\Big)^{\frac12},
    \end{equation}
    while the operatorial norm of $\xi$ is defined by 
    \begin{equation}\nonumber
       |\xi|_{\rm op}:=\sup_{\nu\in\Sn^{
       d-1}}|\xi\nu|.
    \end{equation}
    \item [(c)] If $A\subset \R^d$ is an open set, $\mathcal{A}(A)$ is the collection of all open subsets of $A$. If $A',A''\in\mathcal{A}(A)$, $A'\subset\subset A''$ means that $A'$ is relatively compact in $A''$. The symbol  $\mathcal{B}(A)$ denotes the $\sigma$-algebra of all Borel subsets of $A$. 
    \item [(d)] If $A\in\mathcal{A}(\R^d)$, the space of all $\R^m$-valued bounded Radon measures on $A$ is denoted by $\M_b(A;\R^m)$, the indication of $\R^m$  being omitted if $m=1$. Given a positive $\lambda\in\M_b(A)$ and $\mu \in\M_b(A;\R^m)$, the Radon-Nikod\'ym derivative of $\mu$ with respect to $\lambda$ is denoted by $d\mu/d\lambda$. Given $\mu\in\M(A;\R^{k\times d})$ we denote by $|\mu|$ the total variation computed with respect to the Frobenius norm, while $|\mu|_{\rm op}$ is the total variation computed with respect to the operatorial norm, i.e., the measure defined for every $B\in\mathcal{B}(A)$ by
    \begin{equation*}
   |\mu|_{\rm op}(B):=\sup \sum_{i=1}^n|\mu(B_i)|_{\rm op},
   \end{equation*}
   where the supremum is taken all over $n\in\N$ and finite collections $(B_i)_{i=1}^n$ of pairwise disjoint  relatively compact subsets of  $B$. The $d$-dimensional Lebesgue measure is denoted by $\mathcal{L}^d$, while the $(d-1)$-dimensional Hausdorff measure is denoted by $\mathcal{H}^{d-1}$.
   \item [(e)] Given $A\in\A(\Rd)$, the space of all $\Ld$-measurable functions $u\colon A\to \Rk$ is denoted by $L^0(A;\Rk)$, the indication of $\Rk$ being omitted when $k=1$. We endow this space with the topology induced by the convergence in measure. With this choice of topology, the space $L^0(A;\Rk)$ is  metrisable and separable.
\item [(f)] For every $\Ld$-measurable set $E\subset \Rd$, $\chi_E\colon \Rd \to\R$ is the characteristic function of the set $E$, i.e., $\chi_E(x)=1$ if $x\in E$, $\chi_E(x)=0$ otherwise.
 If $E\subset \Rd$ is an $\Ld$-measurable set with locally finite perimeter (see \cite{maggi2012sets} for the general properties of such sets), we denote by $\partial^*E$ its reduced boundary. We recall that the perimeter of a set relative to $A\in\A (\Rd)$ is defined by 
 \begin{equation*}
     \textup{Per}(u,A):=\hd(\partial ^*E\cap A).
 \end{equation*}
\item [(g)]  If $E\subset \Rd$ is an $\Ld$-measurable set a point $x\in\Rd$ is a point with positive density, i.e., 
\begin{equation*}
    \limsup_{\rho\to 0^+}\frac{\Ld(E\cap B_\rho(x))}{\rho^d}>0,
\end{equation*}
and if 
$u\colon E \to \Rd$ is an an $\Ld$-measurable function, we say that $a\in \Rk$ is the approximate limit of $u$ at the  point $x$, in symbols 
\begin{equation*}
    \underset{\underset{y\in E}{y\to x}}{\text{ap}\lim \,}u(y)=a,
\end{equation*}
if for every $\varepsilon>0$ we have
\begin{equation*}
    \lim_{\rho\to 0^+}\frac{\Ld(\{|u-a|>\varepsilon\}\cap B_\rho(x))}{\rho^d}=0,
\end{equation*}
where $\{|u-a|>\e\}:=\{y\in E\colon |u(y)-a|>\e\}$. Given $A\subset \Rd$ open and $u\in L^0(A;\Rk)$, we denote by $S_u\subset A$ the complement in $A$ of the set of points $x\in A$ such that there exists 
\begin{equation}\nonumber\label{eq:precise representative}
    \widetilde{u}(x):=\underset{\underset{y\in A}{y\to x}}{\text{ap}\lim \,}u(y).
    \end{equation}
If $x\notin S(u)$, we say that $x$ is a point of approximate continuity for $u$ and that $u$ is approximately continuous at $x$.
\item [(l)] Given $A\in\mathcal{A}(\Rd)$  and $u\in L^0(A;\Rk)$  the symbol $\jump{u}$ denotes the jump set of $u$, that is, the set of  $x\in A$ such that 
there exists a triple $(u^+(x),u^-(x),\nu_u(x))\in\Rk\times\Rk\times\mathbb{S}^{d-1}$, with $u^+(x)\neq u^-(x)$, for which, setting 
\begin{equation}\nonumber \label{eq:hyperspaces}
    H^{\pm}_x:=\{y\in A\colon \, \pm(y-x)\cdot \nu_u(x)>0\},
\end{equation}
we have 
\begin{equation*}
      \underset{\underset{y\in H^+_x}{y\to x}}{\text{ap}\lim \,}u(y)=u^+(x)\text{ \,\,\, and }  \,\,\,\underset{\underset{y\in H^-_x}{y\to x}}{\text{ap}\lim \,}u(y)=u^-(x).
\end{equation*}
The triple $(u^+(x),u^-(x),\nu_u(x))$ is unique up to interchanging the roles of $u^+(x)$ and $u^-(x)$ and changing the sign of $\nu_u(x)$. It holds the inclusion $J_u\subset S_u$. Moreover, it was proved in \cite{DelNin} that if $u\in L^1_{\rm loc}$ then $J_u$ is a $(d-1)$-countably rectifiable set (see \cite[Definition 2.57]{AmbFuscPall}).

\item [(m)] Given $A\in\mathcal{A}(\Rd)$, the space of $\Rk$-valued functions with bounded variation on $A$  is denoted by $BV(A;\Rk)$. If $k=1$, we simply write $BV(A)$. We refer the reader to \cite{AmbFuscPall,Evans2015} for a complete introduction to such spaces. 
If $u\in BV(A; \Rk)$ then 
$\hd(S_u\setminus J_u)=0$,  $\jump{u}$ is a $(d-1)$-countably rectifiable set, and for $\hd$-a.e. $x\in\jump{u}$ the vector $\nu_u(x)$ is a measure theoretical unit normal to $\jump{u}$.  For every $x\in\jump{u}$  we set 
\[[u](x):=u^+(x)-u^-(x).\]
Note that a change in the sign $\nu_u(x)$ naturally implies a change of sign in $[u](x)$. 
 
 \item [(n)] If $u\in BV(A;\Rk)$, $Du$ denotes its distributional derivative, which is by definition an $\Rkd$-valued Radon measure. The measure $Du$ can be decomposed as

\[Du=\nabla u\Ld +D^cu+[u]\otimes\nu_u\hd\mres\jump{u},\]
where 
\begin{itemize}
    \item  $\nabla u\in L^1(A;\Rkd)$ is the approximate gradient of $u$, i.e., the unique $\Rkd$-valued function such that for $\Ld$-a.e $x\in A$ we have 
\begin{equation}\label{eq:Approximate Differentaibility}
    \underset{\underset{y\in A}{y\to x}}{\text{ap}\lim \,}\frac{u(y)-\widetilde{u}(x)-\nabla u(x)(y-x)}{|y-x|}=0,
\end{equation}
\item $D^cu$,  the Cantor part of $Du$, is  a measure vanishing on all $B\in\mathcal{B}(A)$ such that $\hd(B)<+\infty$ and which is singular with respect to $\Ld$,
    \item $\otimes$ is the tensor product defined by $(a\otimes b)_{ij}=a_ib_j$ for $a\in\Rk$, $b\in\Rd$, $\hd\mres\jump{u}$ is the restriction of $\hd$ to $J_u$, i.e., the measure on $A$ defined by $(\hd\mres J_u)(B):=\hd(B\cap \jump{u})$ for every $B\in\mathcal{B}(A)$; finally,  $[u]\otimes\nu_u\hd\mres\jump{u}$ is the measure whose density with respect to $\hd\mres\jump{u}$ is given by $[u]\otimes\nu_u$ .  
\end{itemize}
\end{enumerate}

We briefly recall the properties of slicing of functions in $\BVvect$ (see \cite[Section 3.12]{AmbFuscPall}). Given  $\nu\in \mathbb{S}^{d-1}$, the hyperplane through the origin orthogonal to $\nu$ is denoted by  $\Pi_\nu:=\{y\in\Rd:y\cdot\nu=0\}$. For every $y\in\Pi_\nu$, $B\in\mathcal{B}(A)$, we set $B^ \nu_y:=\{t\in\R:\, y+t\nu\in B\}$. For every $u\in \BVvect$, $\nu\in\mathbb{S}^{d-1}$, and  $y\in \Pi_\nu$, $u^\nu_y:A_y^\nu\to\Rk$ is the function defined by 
\begin{equation}\label{eq:slice}    
u^\nu_y(t):=u(y+t\nu)\quad \text{ for every $t\in A^\nu_y$}.
\end{equation}

The following proposition describes some of the well known connections between $BV$ functions and their one-dimensional slices (see \cite[Theorem 3.108]{AmbFuscPall}).

\begin{proposition}\label{prop:Slicing}
    Let $u\in\BVvect$ and let $\nu\in\mathbb{S}^{d-1}$. Then
    \begin{itemize}
        \item [(a)] for $\hd$-a.e $y\in\Pi_\nu$ we have $u^\nu_y\in BV(A^\nu_y;\Rk)$;
        \item [(b)] we have
        \begin{eqnarray*}
           &\displaystyle \nabla(u^\nu_y)=(\nabla u)^\nu_y \quad \text{for $\hd$-a.e. }y\in\Pi_\nu,\\
           &\displaystyle ((D^ju)\nu)(B)=\int_{\Pi_\nu}D^ju^\nu_y(B^\nu_y)\, d\hd(y),\\ 
           &\displaystyle ((D^cu)\nu)(B)=\int_{\Pi_\nu}D^cu^\nu_y(B^\nu_y)\, d\hd(y),
        \end{eqnarray*}
        for every $B\in\mathcal{B}(A)$;
         \item [(c)] if $u_n\rightharpoonup u$ weakly$^*$ in $\BVvect$, then for $\hd$-a.e. $y\in \Pi_\nu$ we have $(u_n)_y^\nu\rightharpoonup u^\nu_y$ weakly$^*$ in $BV(A_y^\nu;\Rk)$.
    \end{itemize}
\end{proposition}

We recall the chain rule in $\BVvect$ for compositions with smooth functions (see \cite[Theorem 3.99]{AmbFuscPall}). For the rest of the section $A\subset\Rd$ will always be a bounded open set.

\begin{theorem}[Chain Rule in $\BVvect$]\label{thm:Chain Rule}
    Let $\varphi\in C^1(\Rk;\R^n)$ and $u\in \BVvect$. Then, setting $v:=\varphi\circ u$, we have $v\in BV(A;\R^n)$ and
    \begin{equation}\label{eq: Chain Rule for BV}
    Dv=\nabla \varphi(u) \nabla u\Ld+\nabla\varphi(\widetilde{u})D^cu+[\varphi(u)]\otimes\nu_u\hd\mres\jump{\phi\circ u}
    \end{equation}
    as Radon measures.
\end{theorem}

 Dal Maso and Toader introduced in \cite{DalToa22} the space $\GBVsscalar$, a subspace of Ambrosio's $GBV(A)$ (see \cite[Section 4.5]{AmbFuscPall}). For the reader's convenience 
 we briefly recall the definition of the spaces $GBV(A)$ and $GBV_\star(A)$ and some relevant properties of such spaces. For every $t\in\R$ and $m>0$, we denote by $\trunc{t}{m}:=(t\land m)\lor (-m)$, where for $a,b\in\R$, the symbols $a\land b$ and $a\lor b$ denote the minimum between $a$ and $b$ and the maximum between $a$ and $b$, respectively. 

\begin{definition}
   Let $u\in L^0(A)$. Then
   \begin{itemize}
        \item[(a)]$u\in GBV(A)$ if and only if
    $u^{(m)}\in BV_{\rm loc}(A)$ for every $m>0$;
    \item [(b)] $u\in \GBVsscalar$ if and only if  $u^{(m)}\in BV(A)$ for every $m>0$ and 
    there exists $M>0$ such that
\begin{equation}\label{eq:boundsinglecomponent}
\sup_{m>0}\int_A|\nabla\trunc{u}{m}|\, dx+|D^c\trunc{u}{m}|(A)+\int_{{J}_{\trunc{u}{m}}}|[\trunc{u}{m}]|\land 1\,d\mathcal{H}^{d-1}\leq M.
\end{equation}
\end{itemize}
\end{definition}

The following proposition collects the main properties of the spaces $GBV(A)$ and $GBV_\star(A)$.
\begin{proposition}\label{prop: Properties of GBVs scalar}
    Let $u\in GBV(A)$. Then
    \begin{itemize}
        \item [(a)] for $\hd$-a.e $x\in A\setmeno\jump{u}$ there exists finite
        \[\widetilde{u}(x):=\underset{y\to x}{\textup{ap}\lim \,}u(y);\]
        moreover $u^+$ and $u^-$ are finite for $\hd$-a.e. $x\in J_u;$
        \item [(b)] $u$ is approximately differentiable, i.e., there exists a Borel function $\nabla u: A\to\Rd$ such that formula \eqref{eq:Approximate Differentaibility} holds for $\Ld$-a.e. $x\in A$; moreover, for every $m>0$ we have
        \[\nabla u(x)=\nabla \trunc{u}{m}(x)\quad \text{for $\Ld$-a.e. }x\in \{x\in A\colon |u(x)|\leq m\};\]
        \item [(c)] 
        there exists a unique positive measure $|D^cu|\in\mathcal{M}_b(A)$ such that
        \[
        |D^cu|(B)=\sup_{m>0}|D^cu^{(m)}|(B)\quad \text{for every }B\in\mathcal{B}(A);
        \]
        \item [(d)] for every $m>0$ we have $\jump{\trunc{u}{m}}\subset\jump{u}$ up to an $\hd$-negligible set and $|[\trunc{u}{m}]|\leq|[u]|$ $\hd$-a.e. in $\jump{\trunc{u}{m}}\cap\jump{u}$. Moreover, for $\hd$-a.e. $x$ in $\jump{u}$ there exists $m_x\in\mathbb{N}$ such that $x\in \jump{u^{(m)}}$ for every $m\in\mathbb{N}$ with $m\geq m_x$, and $[u^{(m)}](x)\to [u](x)$ as $m\rightarrow+\infty$.
        
          \end{itemize}
\vspace{0.2 cm}
          \noindent If, in addition, $u\in \GBVsscalar$, then 
        \begin{itemize}
            \item [(e)] there exists a unique Radon measure $D^cu\in\mathcal{M}_b(A;\Rd)$ such that for every $m>0$ we have $D^cu(B)=D^c\trunc{u}{m}(B)$ for every $B\subset\{x\in A\colon \, \widetilde{u}(x) \text{ exists and }|\widetilde{u}(x)|\leq m\}$ and $D^cu(B)=0$ for every $B\in\mathcal{B}(A)$ with $\hd(B\setmeno\jump{u})=0;$ moreover, for every $B\in\mathcal{B}(A)$ we have
        \begin{eqnarray} \nonumber \label{eq: Definition of Cantor for GBVs}
       & \displaystyle D^c(B)=\lim_{m\to \infty} D^c\trunc{u}{m}(B),\\\nonumber 
       & \label{eq:Definition of Total Variation of the cantor part in GBV}\displaystyle |D^cu|(B)=\sup_{m>0}|D^cu^{(m)}|(B).
        \end{eqnarray}
  \end{itemize}
\end{proposition}
For the proofs of (a)-(d) we refer the reader to \cite[Theorem 4.34]{AmbFuscPall}, while for the proof of (e) we refer to \cite[Theorem 2.7, Propositions 2.9, 3.3]{DalToa22}.

We conclude the current section recalling the vector space properties of  $\GBVsscalar$.

\begin{proposition}[{\cite[Theorem 3.9]{DalToa22}}]\label{prop:GBVsscalar is a v.space }
    $\GBVsscalar$ is a vector space.  Moreover, for every $u,v\in \GBVsscalar$ and  $\lambda\in \R$ we have
\begin{eqnarray}
&\label{eq: nabla somme scalare}\displaystyle\nabla(u+v)=\nabla u+\nabla v, \,\,\, \nabla(\lambda u)=\lambda \nabla u \,\,\, \Ld\text{-$a.e. $ in } A,\\
    &\displaystyle \label{eq:cantor somme scalare}D^c(u+v)=D^cu+D^cv, \,\,\,D^c(\lambda u)=\lambda D^cu \,\,\, \text{ on }A,  \\
    &\displaystyle\label{De Gamma Somme scalare}[u+v]=[u]+[v]\,\,\text{ and }\,\,[\lambda u]=\lambda [u] \,\,\, \text{ $\hd$-a.e. in  }A.
\end{eqnarray}
\end{proposition}

\section{The space \texorpdfstring{$\GBVsvector$}{e}}

In this section we give the Definition of the space $GBV_\star(A;\Rk)$, and discuss some of its properties. 

Throughout the rest of the paper $A\subset\Rd$ will be a bounded open set.
\begin{definition}\label{def:GBVstar}
   Let $u\in\! L^0(A;\Rk)$. We say
    $u\in GBV_\star(A;\Rk)$ if  $u_i\in GBV_\star(A)$ for every $i=1,...,k$. 
\end{definition}
\begin{remark}
   The space $GBV(A;\Rk)$ of generalised functions of  bounded variation contains the space $GBV(A)^k$ (strictly if $k>1$, see \cite[Remark 4.27]{AmbFuscPall}) of $\Rk$-valued functions whose components are in  $GBV(A)$. It is easy to see that $GBV_\star(A;\Rk)\subset GBV(A)^k\subset GBV(A;\Rk)$, the inclusions being in general strict.
\end{remark}

It follows immediately from Proposition \ref{prop:GBVsscalar is a v.space } that $\GBVsvector$ is a vector space.

\begin{proposition}\label{prop:GBVsvector is a vector space}
  $GBV_\star(A;\Rk)$ is a vector space.
\end{proposition}
If $m>0$ and $u\in L^0(A;\Rk)$, we set $u^{(m)}:=((u_1\land m)\lor(-m),...,(u_k\land m)\lor(-m))$.

\begin{definition}\label{def:Cantor part}
    Let $u\in \GBVsvector$. The measure $D^cu\in \mathcal{M}_b(A,\Rkd)$ is the matrix-valued measure whose $i$-th row is defined for every $B\in\mathcal{B}(A)$ by 
    \[(D^cu(B))_i:= D^cu_i(B).\]
\end{definition}
The measure $D^cu$ enjoys the same properties 
of its scalar counterpart (see Proposition \ref{prop: Properties of GBVs scalar}(e)).
\begin{lemma}\label{lemma:Dcu is a sup} 
    Let $u\in\GBVsvector$. Then 
    \begin{itemize}
        \item[(a)] for every $m>0$  we have $D^cu(B)=D^c\trunc{u}{m}(B)$ for every $B\subset\{x\in A\colon \, \widetilde{u}(x) \text{ exists and }|\widetilde{u}(x)|\leq m\};$
        \item[(b)] $D^cu(B)=0$  for every $B\in\mathcal{B}(A)$ with $\hd(B\setmeno\jump{u})=0$;
        \item [(c)] setting $\widetilde{A}:=\{x\in A:\widetilde{u}(x) \text{ exists}\}$, we have $\hd( \widetilde{A}\setminus J_u)=0.$
    \end{itemize}
    Moreover,
    \begin{eqnarray}
    &\displaystyle\label{Cantor derivative vector}
    D^cu(B)=\lim_{m\to+\infty} D^cu^{(m)}(B),\\\label{eq:Total Variation Vectorial}
    &\displaystyle |D^cu|(B)=\sup_{m>0}|D^cu^{(m)}|(B),
    \end{eqnarray}
    for every $B\in\mathcal{B}(A).$
\end{lemma}
\begin{proof}
    The proof of (a) and (b) is a simple consequence of Proposition \ref{prop: Properties of GBVs scalar} and of the definition of $D^cu$.  To prove (c), it is enough to note that for every $i\in\{1,...,k\}$ we have  $\widetilde{A}=\cap_{i=1}^k\widetilde{A}_i$, where $\widetilde{A}_i:=\{x\in A: \widetilde{u}_i(x) \text{ exists}\},$ and use (a) of Proposition \ref{prop: Properties of GBVs scalar}.

    We are left with proving  that for any $B\in \B(A)$ equalities \eqref{Cantor derivative vector} and \eqref{eq:Total Variation Vectorial} hold.  To this aim, let us fix $B\in\B(A)$. For every $m\in\N$ we set
        \[
    B_m:=\bigcap_{i=1}^k\{x\in B\colon \widetilde{u}_i(x) \text{ exists and } |\widetilde{u}_i(x)|\leq m\}.\]
    By Proposition \ref{prop: Properties of GBVs scalar} we see that $|D^cu|(B\setminus\bigcup_{m\in\N}B_m)\leq \sum_{i=1}^k|D^cu_i|(B\setminus\bigcup_{m\in\N}B_m)=0$. Hence, for every $\e>0$ there exists $m'\in\N$  such that  $|D^cu|(B\setminus B_{m})\leq \e$ for every $m\geq m'$. Using again Proposition \ref{prop: Properties of GBVs scalar}, by definition of $B_m$ we have 
   $D^cu_i(B_m)=D^cu^{(m)}_i(B_m)$ for every $i\in\{1,...,k
   \}$. Thus for every $m\geq m'$ we have
   \begin{equation*}
       |D^cu(B)-D^cu^{(m)}(B)|\leq |D^cu|(B\setminus B_m)\leq \e.
   \end{equation*}
   Letting $\e\to 0^+$, we obtain \eqref{Cantor derivative vector}. 

   The proof of equality \eqref{eq:Total Variation Vectorial} follows by similar  arguments. This concludes the proof.
\end{proof}

We now show that $\GBVsvector$ is well-behaved under linear changes of coordinates.

\begin{proposition}
    Let $u\in \GBVsvector$ and let  $C\in \R^{k\times k}$ be  an invertible matrix. Then $v:=Cu\in GBV_\star(A;\R^k)$ and
    \begin{eqnarray}
    &\label{eq:grandient composed matrix }\displaystyle \nabla v=C\nabla u \,\, \text{$\Ld$-a.e. in }A,\,\,
    \\ & \label{eq: Cantor composed matrix} \displaystyle D^cv(B)=CD^cv(B) \text{ for every $B\in\mathcal{B}(A)$},\\
    & \label{eq: Jump composed matrix}
  [v]=C[u]\quad \text{$\hd$-a.e. in } A.
     \end{eqnarray}
\end{proposition}

\begin{proof}
    The fact that $v\in GBV(A; \R^k)$ is a  consequence of Proposition \ref{prop:GBVsvector is a vector space}, while \eqref{eq:grandient composed matrix }, \eqref{eq: Cantor composed matrix}, and \eqref{eq: Jump composed matrix} are a consequence of \eqref{eq: nabla somme scalare}, \eqref{eq:cantor somme scalare}, and \eqref{De Gamma Somme scalare}.
\end{proof}
 The following result, which generalises \cite[Proposition 3.4]{DalToa22}, characterises $\GBVsvector$ as a subspace of $GBV(A)^k$.
\begin{proposition}\label{prop:Subspace of GBV}
Let $u\in GBV(A)^k$. Then $u$ belongs to $\GBVsvector$ if and only if $u$ satisfies
    \begin{eqnarray}
        &\label{eq:nabla L1}\displaystyle \nabla u\in L^1(A;\Rkd),\\
        &\label{eq:Cantor is bounded}\displaystyle \sum_{i=1}^k|D^cu_i|\in\mathcal{M}_b(A),\\
        &\label{eq: Jump 1 e integralone are controlled}\displaystyle \int_{\jump{u}\setminus \jump{u}^1}|[u]|\, d\hd<+\infty \text{ and }\hd(\jump{u}^1)<+\infty,
    \end{eqnarray}
    where $\jump{u}^1:=\{x\in\jump{u}:|[u](x)|\geq 1\}$.
\end{proposition}
\begin{proof}
The only if part is a consequence of Proposition \ref{prop: Properties of GBVs scalar} and of the equality
\begin{equation}\label{eq:trivial}
    \int_{\jump{u}}|[u]|\land 1\, d\hd=\int_{\jump{u}\setminus \jump{u}^1}|[u]|\, d\hd+\hd(\jump{u}^1).
\end{equation}

To prove the if part,  let $u\in GBV(A)^k$ be a function satisfying \eqref{eq:nabla L1}, \eqref{eq:Cantor is bounded}, and \eqref{eq: Jump 1 e integralone are controlled}. Let us fix  $i\in\{1,...,k\}$. First, we show that for every $m>0$ we have $u_i^{(m)}\in BV(A)$.
Since $\|u_i^{(m)}\|_{L^\infty(A)}\leq m$ and $u$ satisfies \eqref{eq:nabla L1}  and  \eqref{eq:Cantor is bounded}, we only have to show that  
\begin{equation*}
    \int_{J_{u^{(m)}_i}}|[u_i]|\,d\hd<+\infty.
\end{equation*}
Taking advantage of \eqref{eq: Jump 1 e integralone are controlled}, we see that
\begin{align*}
\displaystyle\int_{\jump{u_i^{(m)}}}|[u_i^{(m)}]|\, d\hd&=\int_{\jump{u_i^{(m)}}\setminus \jump{u_i^{(m)}}^1}\!\!\!\!\!\!\!\!\!|[u_i^{(m)}]|\, d\hd+\int_{ \jump{u_i^{(m)}}^1}|[u_i^{(m)}]|\, d\hd\\
    &\leq \int_{\jump{u}\setminus \jump{u}^1}|[u]|\,d\hd+(2m+1)\hd(\jump{u}^1)<+\infty.
\end{align*}
Hence, $u\in BV(A)$. Finally, by  \eqref{eq:trivial} we see that for every $m>0$
\[
\int_{\jump{u_i^{(m)}}}\!\!\!\!\!\!\!\!\!\!\!\!|[u_i^{(m)}]|\land 1\, d\hd\leq\int_{\jump{u}\setminus \jump{u}^1}|[u]|\, d\hd+\hd(\jump{u}^1).\]
Combining this last inequality  with \eqref{eq:nabla L1}, \eqref{eq:Cantor is bounded} and \eqref{eq: Jump 1 e integralone are controlled}, we deduce that
$u_i\in\GBVsscalar$, so that from the arbitrariness of $i$, we deduce $u\in\GBVsvector$, concluding the proof.\end{proof}

 \begin{remark}
    It follows immediately from Definition \ref{def:GBVstar}, Lemma \ref{lemma:Dcu is a sup}, and Proposition \ref{prop:Subspace of GBV}  that $u\in GBV_\star(A;\Rk)$ if and only if there exists a constant $M>0$ such that for every $m>0$ the function $u^{(m)}$ belongs to $BV(A;\Rk)$ and 
    \begin{equation}\nonumber \label{eq:defvectgbv}
   \sup_{m>0}\int_A|\nabla u^{(m)}|\, dx+|D^cu^{(m)}|(A)+\int_{J_{u^{(m)}}}|[u^{(m)}]|\land 1\,d\mathcal{H}^{d-1}\leq M.
   \end{equation}
\end{remark}

 In \cite[Lemma 2.10]{Alicandro-Focardi}, Alicandro and Focardi are able to define a measure which they call $D^cu$ for every $u\in GBV(A)^k$ such that $|D^cu|(A)<+\infty$. To define such a measure, they deal with {\it smooth} truncations. In analogy with their approach, we characterise $GBV_\star(A;\Rk)$ by means of composition with smooth functions. To this aim, we introduce the following functions. 
Given  a positive constant $\sigma>2$, we fix a  smooth radial function $\psi\in C^\infty_c(\Rk;\Rk)$ satisfying
\begin{equation}\nonumber \label{eq:defpsi}
    \begin{cases}
        \psi(y)=y&\text{ if }|y|\leq 1,\\
        \psi(y)=0&\text{ if }|y|\geq \sigma,\\
        |\psi(y)|\leq \sigma,\\
        \textup{Lip}(\psi)=1.
    \end{cases}
\end{equation}
Given $R>0$, we set
\begin{equation}\label{def:psiR}\psi_{R}(y):=R \psi\bigg(\frac{y}{R}\bigg)\quad \text{ for every $y\in\Rk$}.
\end{equation}
Observe that the function  $\psi_R$ satisfies
\begin{subequations}
  \begin{empheq}[left=\empheqlbrace]{align}
\label{eq:PsiR}  &\psi_R(y)=y&  &\hspace{-2cm}\text{for }y\in\{|y|\leq R\},\\
 \label{eq:PsiR0}&\psi_R(y)=0& &\hspace{-2cm}\text{for }y\in \{|y|\geq \sigma R\},\\ 
 \label{eq:PsiRbounded}&|\psi_R(y)|\leq \sigma R&\\
\label{eq:PsiRLip}&\text{Lip}(\psi_R)\leq 1.&
  \end{empheq}
\end{subequations}

The following proposition characterises $GBV_\star(A;\Rk)$ in terms of composition with smooth functions. 
\begin{proposition}\label{prop:characterisation}
   For every  $u\in GBV_\star(A;\Rk)$  there is a constant $C_u>0$ such that for every Lipschitz function $\phi\colon \Rk\to\Rk$ with compact support the function  $v:=\phi\circ u$ is in  $BV(A;\Rk)$ and satisfies the following inequality
\begin{equation}\label{eq:def GBVstar}
    \int_A|\nabla v|\, dx+|D^cv|(A)+\int_{J_{v}}|[v]|\land 1\,d\mathcal{H}^{d-1}\leq C_u(\textup{Lip}(\phi)\lor 1).
\end{equation}
Conversely, if $u\in L^0(A;\Rk)$ and there is a constant $C_u>0$ such that for every integer $R>0$ the function $\psi_R\circ u$ belongs to $BV(A;\Rk)$ and inequality \eqref{eq:def GBVstar} holds with $v=\psi_R\circ u$ and $\phi=\psi_R$, then $u\in GBV_\star (A;\Rk)$.
\end{proposition}

\begin{proof}
We begin by proving that if $u\in GBV_\star(A;\Rk)$ there is a constant $C_u>0$ such that \eqref{eq:def GBVstar} is satisfied. Thanks to the elementary inequalities
\begin{eqnarray*}
&&\displaystyle \label{eq:elementary 1}|\xi|\leq \sum_{j=1}^k|\xi_j|\quad\text{for every $i\in\{1,...,k\}$ and $\xi\in\Rkd$},\\
&&\displaystyle \label{elementary 2} |[\zeta]|\land 1\leq \sum_{j=1}^k(|[\zeta_j]|\land 1)\quad \text{for every $i\in\{1,...,k\}$ and $\zeta\in\Rk$},
\end{eqnarray*}
proving \eqref{eq:def GBVstar} is equivalent to proving that \eqref{eq:def GBVstar} holds with $v$ replaced by $v_i$. 

We prove first the result for $\phi\in C^1_c(\Rk;\Rk)$.
We set $v=\phi\circ u$ and $K:=\text{supp}(\phi)$. Note that if $m> \max_{y\in K}\lvert y \rvert$, then $v=\phi(\trunc{u}{m})$ $\Ld$-a.e. in $A$. Since $u\in GBV_\star(A;\Rk)$, then $\trunc{u}{m}\in BV(A;\Rk)$. By the chain rule \eqref{eq: Chain Rule for BV} we then infer that $v_i\in BV(A)$ for every $i\in\{1,...,k\}$. 

We claim that up to an $\hd$-negligible set $J_{v_i}\subset J_{u_i^{(m)}}$. Indeed, by \cite[Proposition 1.1 (iii)]{Ambrosio1990ExistenceTF} every point $x\in A$ of approximate continuity for $u_i$ is also a point of approximate continuity for $v_i$, so that we have the inclusion $S_{v_i}\subset S_{u_{i}^{(m)}}$ and by (m) of Section 2 we have  $\hd(S_{u_i^{(m)}}\setminus J_{u_i^{(m)}})=\hd(S_{v_i}\setminus J_{v_i})=0$, whence the claim.

 Using the Chain Rule \eqref{eq: Chain Rule for BV}, together with Proposition \ref{prop: Properties of GBVs scalar} and  inequality \begin{equation*}
     |[\phi_i(\zeta)]|\land 1\leq (\textup{Lip}(\phi)\lor 1)\sum_{i=1}^k (|[\zeta_i]|\land 1),
 \end{equation*} we see that
\begin{eqnarray*}
\int_A|\nabla v_i|\, dx\!\!\!\!\!\!\!\!\!\!\!&&\,+|D^cv_i|(A)+\int_{J_{v_i}}|[v_i]|\land 1\, d\mathcal{H}^{d-1}\\
&&\hspace{-1.90 cm}=\int_A|\nabla \phi_i(\trunc{u}{m})\nabla\trunc{u}{m}|\,dx\!+\!\int_A|\nabla\phi_i(\widetilde{\trunc{u}{m}})|\, d|D^c\trunc{u}{m}|+\int_{J_{\phi_i(\trunc{u}{m})}}\!\!\!\!\!\!\!\!\!\!\!\!\!\!\!\!\!\!|[\phi_i(u^{{(m)}})]|\,d\mathcal{H}^{d-1}\\ 
&&\hspace{-1.90 cm}\leq (\textup{Lip}(\phi)\lor 1)\sum_{i=1}^k\Big(\int_A|\nabla u^{(m)}_i|\, dx+|D^c\trunc{u_i}{m}|(A)+\int_{J_{\trunc{u_i}{m}}}\!\!\!\!\!|[u_i^{(m)}]|\land 1\,d\mathcal{H}^{d-1}\Big).
\end{eqnarray*}
Since $u\in GBV_\star(A;\Rk)$ there exists a constant $M>0$ such that    \eqref{eq:boundsinglecomponent} holds  for every $i\in\{1,...,k\}$ and from the previous inequality  it follows  that 
\begin{equation*}
    \int_A|\nabla v_i|\, dx+|D^cv_i|(A)+\int_{J_{v_i}}|[v_i]|\land 1\, d\mathcal{H}^{d-1}\leq kM(\textup{Lip}(\phi)\lor 1),
\end{equation*} so that \eqref{eq:def GBVstar} is proved for $\phi \in C^1_c(\Rk;\Rk)$, with $C_u=k^2M$.

If $\phi$ is Lipschitz with compact support, then there exists a sequence  $(\phi_n)_n\subset C^1_c(\Rk;\Rk)$ such that $\phi_n\to \phi$ uniformly on $\Rk$ and such that $\textup{Lip}(\phi_n)\leq \textup{Lip}(\phi)$. In particular, $\phi_n\circ u\to \phi\circ u$ in $L^1(A;\Rk)$. We claim that
\begin{equation}\label{eq:primo claim di boundedness}
    \sup_{n\in\N} |D(\phi_n\circ u)|(A)<+\infty.
\end{equation}
Indeed, from \eqref{eq:def GBVstar} applied to $v=\phi_n\circ u$,  we deduce that to prove the claim it is enough to check that 
\[
\sup_{n\in\N}\int_{J_{\phi_n\circ u}}|[\phi_n\circ u]|\,d\hd<+\infty.
\]
Using again \eqref{eq:def GBVstar} with $v$ replaced by $\phi_n\circ u$, and recalling that $\textup{Lip}(\phi_n)\leq \textup{Lip}(\phi)$ and $\phi_n\to \phi$ uniformly, for $n$ large enough we obtain that
\begin{align}\nonumber 
    \int_{J_{\phi_n\circ u}}\!\!\!\!\!\!|[\phi_n\circ u]|\,d\hd & \leq \int_{J_{\phi_n\circ u}\setminus J^1_{\phi_n\circ u}}\!\!\!\!\!\!\!\!\!\! |[\phi_n\circ u]|\,d\hd +2\max_{y\in\Rk}|\phi_n(y)|\hd(J^1_{\phi_n\circ u})\\\notag
    &\leq(1+4\max_{y\in\Rk}|\phi(y)| )\int_{J_{\phi_n\circ u}}|[\phi_n\circ u]|\land 1\,d\hd\\
    &\leq (1+4\max_{y\in\Rk}|\phi(y)| )C_u(\text{Lip}(\phi)\lor 1),
\end{align}
where $J^1_{\phi_n\circ u}$ is defined as in the statement of Proposition \ref{prop:Subspace of GBV}. This proves the claim.

By the lower semicontinuity of the total variation with respect to the $L^1$-convergence it follows that $\phi\circ u\in BV(A;\Rk)$, while from \cite[Theorem 2.1]{Braides1995} and $\textup{Lip}(\phi_n)\leq   \textup{Lip}(\phi) $ it follows that
\begin{equation*}
     \int_A|\nabla (\phi\circ u)_i|\, dx+|D^c(\phi\circ u)_i|(A)+\int_{J_{(\phi\circ u)_i}}\!\!\!|[(\phi\circ u)_i]|\land 1\, d\mathcal{H}^{d-1}\leq C_u(\textup{Lip}(\phi)\lor 1),
\end{equation*}
so that \eqref{eq:def GBVstar} is proved.

We now show that also the converse holds.  To this aim, for every $m\in\N$ let $(\phi^m_n)_n$ be a sequence of $C^1$-functions from $\Rk$ to $\Rk$ converging to $\xi\mapsto\xi^m$ uniformly and such that $\text{Lip}(\phi^m_n)\leq1$ and $\|\phi_n^m\|_{L^\infty}\leq 2m$ for every $n\in\N$. For  $R>0$ and  $m,n\in\N$ we set $v^m_{R,n}:=\phi^m_n\circ(\psi_R\circ (u))$. It follows immediately from \eqref{eq:PsiR} that for every $m,n>0$ we have
\begin{equation}\label{eq:pointwise limit in proof equiv def}
    \lim_{R\to+\infty}v^m_{R,n}(x)=(\phi^m_n\circ u)(x)\quad\quad\text{for }\mathcal{L}^d\text{-a.e. }x\in A.
\end{equation}
Since $\|v^m_{R,n}\|_{L^\infty(A;\Rk)}\leq 2m$ for every $n\in\N$, it follows from \eqref{eq:pointwise limit in proof equiv def} that 
\begin{equation}\label{eq:vrm converges in L^1}
    v^m_{R,n} \text{ converges to } (\phi^m_n\circ u) \text{ in } L^1(A;\Rk)\text{ as $R\to+\infty$}.
\end{equation}
 By definition of $\phi^m_n$ we also have
\begin{equation}\label{convergence in n}
   \phi^m_n\circ u \text{ converges to } u^{(m)}\text{ in } L^1(A;\Rk)\text{ as $n\to+\infty$}.
\end{equation}
By the Chain Rule \eqref{eq: Chain Rule for BV}, we have that
\begin{gather*}
    \nabla v^m_{R,n}=\nabla\phi _n^m(\psi_R\circ u)\nabla(\psi_R\circ u) \,\, \Ld\text{-a.e. in } A,\\
   D^c v^m_{R,n}=\nabla\phi _n^m(\psi_R\circ \widetilde{u})D^c(\psi_R\circ u)\,\, \text{ as Borel measures in }A,\\
   [v^m_{R,n}]=[\phi^m_{n}(\psi_R(u^+))-\phi^m_{R,n}(\psi_R(u^+))] \,\, \hd\text{-a.e in }J_{v^m_{R,n}},
\end{gather*}
 so that, using $\text{Lip}(\phi^m_n)\leq 1$, we have 
 \begin{gather*}
     |\nabla v^m_{R,n}|\leq |\nabla(\psi_R\circ u)|\,\, \Ld\text{-a.e. in } A,\\
   |D^c v^m_{R,n}|(B)\leq |D^c(\psi_R\circ u)|(B)\,\,\text{ for every }B\in\B(A),\\
   |[v^m_{R,n}]|\leq|[\psi_R\circ u]|\,\,\,\, \hd\text{-a.e in }J_{v^m_{R,n}}.
 \end{gather*}
From these  inequalities, together with \eqref{eq:def GBVstar}, for every $R>0$  and  $m,n\in\N$ we obtain that
\begin{equation}\label{eq:sium}
    \int_A|\nabla v^m_{R,n}|\,dx+|D^cv^m_{R,n}|(A)+\int_{J_{v^m_{R,n}}}\!\!\!\!|[v^m_{R,n}]|\land 1\,d\mathcal{H}^{d-1}\leq C_u,
\end{equation} for a constant $C_u>0$, independent of $R>0$, $n$, and of $m$.

We claim that for every $m>0$ 
\begin{equation}\label{eq:psiRtruncatedisBV}
\sup_{n\in\N}\sup_{R>0}|Dv^m_{R,n}|(A)<+\infty .
\end{equation}
To show this, it is sufficient to prove that 
\[\sup_{n\in\N}\sup_{R>0}\int_{J_{v^m_{R,n}}}|[v^m_{R,n}]|\,d\mathcal{H}^{d-1}<(4m+1)C_u.\]
By \eqref{eq:sium} and $\|v^m_{R,n}\|_{L^\infty(A;\Rk)}\leq 2m$ we have 
\begin{align*}
    \int_{J_{v^m_{R,n}}}\!\!\!\!\!\!\!\!\!\!|[v^m_{R,n}]|\,d\mathcal{H}^{d-1}
    &\leq\int_{J_{v^m_{R,n}}\setminus J^1_{v^m_{R,n}}}\!\!\!\!\!|[v^m_{R,n}]|\,d\mathcal{H}^{d-1}+4m\mathcal{H}^{d-1}(J^1_{v_{R,n}^m})\\
    &= (4m+1)\int_{J_{v^m_{R,n}}}|[v^m_{R,n}]|\land 1\,d\hd\leq (4m+1)C_u,
\end{align*}
where  the last inequality follows from \eqref{eq:sium}. This proves \eqref{eq:psiRtruncatedisBV} and shows that  $v^m_{R,n}\in BV(A;\Rk)$ for every $R>0$ and $n\in\N$. 

Thanks to \eqref{eq:vrm converges in L^1}, \eqref{convergence in n}, and  \eqref{eq:psiRtruncatedisBV}, the lower semicontinuity of the Total Variation with respect to the $L^1$-convergence implies that $\trunc{u}{m}\in BV(A;\Rk)$ for every $m>0$. 

Finally, in light of \eqref{eq:vrm converges in L^1}, \eqref{convergence in n}, \eqref{eq:psiRtruncatedisBV}, and $\text{Lip}(\phi^m_n)\leq 1$, we may apply \cite[Theorem 2.1]{Braides1995} twice to conclude that for every $i\in\{1,...,k\}$ we have
\begin{align*}
 &\hspace{-1.5 cm}\int_A|\nabla u_i^{(m)}|\,dx+|D^c u_i^{(m)}|(A)\int_{J_{u_i^{(m)}}}|[u_i^{(m)}]|\land 1\,d\mathcal{H}^{d-1}
\\
&\leq\liminf_{n\to+\infty}\liminf_{R\to+\infty}\Big(\int_A|\nabla (v_{R}^m)_i|\,dx+|D^c(v^m_R)_i|(A)\!\\&\hspace{3.5 cm}+\int_{J_{(v_{R}^m)_i}}\!\!\!\!\!\!\!\!\!\!|[(v_{R}^m)_i]|\land 1\,d\hd
\!\Big)\leq C_u,
\end{align*}
 which implies that $u\in GBV_\star(A;\Rk)$, concluding the proof.
\end{proof}

\begin{remark}
To guarantee that $u\in GBV_\star(\Omega;\Rk)$, it is crucial that the bound on the right-hand side of \eqref{eq:def GBVstar} does not depend on the support of $\phi$, as it can be seen considering a function suggested by \cite[Remark  4.9]{Pallara90}. Indeed, Let $A=(-1,1)$ and $u=\text{sign}(\sin(\frac{\pi}{x}))/x$. It is immediate to check that $J_u=\{1/n\}_{n\in \Z}$ and $|[u](1/n)|=2n$ for every $n\in\Z$. For every $R>0$, consider a function $\phi\in C^1_c(\R;\R)$, with supp$(\phi)=[-R;R]$ and  $\phi(y)=y$ for $|y|\leq R/2$. We set $v:=\phi\circ u$ and note that $J_{v}= \{1/n\colon n\in\Z \text{ and } |n|\leq \lfloor R\rfloor\}$, where $\lfloor\,\cdot\,\rfloor$ denotes the floor function. Then 
\begin{equation*}
    \int_{J_v}|[v]|\land 1\,d\mathcal{H}^0\geq\mathcal{H}^0(\{1/n\colon n\in\Z \text{ and } |n|\leq \lfloor R/2\rfloor\}),
\end{equation*}
so that we immediately get the  lower bound 
\begin{equation*}
\int_A|\nabla v|\,dx+  \int_{J_v}|[v]|\land 1\,d\mathcal{H}^0\geq \mathcal{H}^0(\{1/n\colon n\in\Z \text{ and } |n|\leq \lfloor R/2\rfloor\}),
\end{equation*}
which blows up as $R\to +\infty $, so that $u\notin GBV_\star(A)$.  
\end{remark}

\begin{remark}\label{re:Vectorholds with n generic}
     Let $n\in\N$. Following the lines of the proof of Proposition \ref{prop:characterisation}, it can be shown that $u\in\GBVsvector$ if and only if there exists a constant $C_u>0$ such that for every $\phi\in C^1_c(\Rk,\R^n)$ the function $\phi\circ u\in BV(A;\R^n)$ and 
     \begin{equation*}
         \int_{A}|\nabla (\phi\circ u)|\, dx+|D^c(\phi\circ u)|(A)+\int_{J_{\phi\circ u}}|[\phi\circ u]|\, d\hd\leq C_u(\textup{Lip}(\phi)\lor 1).
     \end{equation*}
     In particular, we can use $n=1$.
\end{remark}

The next proposition collects some properties of the approximate gradient and of the jump part of compositions of functions in $GBV_\star(A;\Rk)$ with the function $\psi_R$ defined in \eqref{eq:defpsi}.
\begin{proposition}\label{prop:psi sui salti}
    Let $u\in GBV_\star(A;\Rk)$. Then for every $R>0$ we have
    \begin{itemize}
        \item[(i)] $\nabla (\psi_R\circ u)=\nabla u$ $\Ld$-a.e. in $\{x\in A\colon |u(x)|\leq R\}$;
        \item[(ii)] $\jump{\psi_R\circ u}\subset\jump{u}$ up to an $\hd$-negligible set  and $|[\psi_R\circ u]|\leq|[u]|$ on $\jump{\psi_R\circ u}\cap\jump{u}$. Moreover, for $\hd$-a.e $x$ in $\jump{u}$ and every $R>|u^+(x)|\lor|u^-(x)|$  we have $|[\psi_R\circ u](x)|=|[u](x)|$.
    \end{itemize}
\end{proposition}
\begin{proof}
    The proof of (i) is a  consequence of \cite[Proposition 1.2]{Ambrosio1990ExistenceTF} and the fact that $\nabla\psi_R$ is the identity on $\{x\in A\colon |u(x)|\leq R\}$. 

    The inclusion $\jump{\psi_R\circ u}\subset\jump{u}$ up to an $\hd$-negligible set  was already proved in the first part of the  proof of Proposition \ref{prop:characterisation}, while the inequality $|[\psi_R\circ u]|\leq|[u]|$ on $\jump{\psi_R\circ u}\cap\jump{u}$ is a consequence of \eqref{eq:PsiRLip}.   The fact that for $\hd$-a.e $x$ in $\jump{u}$ and every $R>|u^+(x)|\lor|u^-(x)|$  we have $|[\psi_R\circ u](x)|=|[u](x)|$, follows immediately from \eqref{eq:PsiR}, and the fact that $[\psi_R\circ u]=\psi_R(u^+)-\psi_R(u^-)$.
\end{proof}

In the following proposition we study the relation between the Cantor part of $u\!\in \!GBV_\star(A;\Rk)$ and the Cantor part of smooth truncations of $u$. As a consequence, we derive the equivalence of our Definition \ref{def:Cantor part} and the one given by Alicandro and Focardi in \cite{Alicandro-Focardi}.
\begin{proposition}\label{prop:properties of CantorPart}
     Let $u\in GBV_\star(A;\Rk)$ and $\phi\in C^1_c(\Rk;\Rk)$. Then 
    \begin{itemize}
        \item[(i)]  $D^c(\phi \circ u)=\nabla \phi(\widetilde{u})D^cu$ as Radon measures on $A$;
        \item [(ii)] we have
        \begin{eqnarray}
            &\displaystyle\label{eq:limits cantor 1} D^cu(B)=\lim_{R\to+\infty}D^c(\psi_R\circ u)(B) \quad \text{ for every $B\in\mathcal{B}(A)$,}\\
            &\displaystyle \label{eq:limits cantor 2} |D^cu|(B)=\lim_{R\to+\infty}|D^c(\psi_R\circ u)|(B) \quad \text{ for every $B\in\mathcal{B}(A)$.}
        \end{eqnarray}
        \item [(iii)]  for every $R>0$ we have 
        \begin{equation}\label{eq:prima propr cantor}
            \frac{dD^c(\psi_{R}\circ u)}{d|D^c(\psi_{R}\circ u)|}=\frac{dD^cu}{d|D^cu|} \quad \text{$|D^cu|$-a.e. in $\ A_R$},
        \end{equation} 
          where $A_R:=\{x\in A\colon \widetilde{u}(x)\text{ exists and } |\widetilde{u}(x)|\leq R\}$.
        As a consequence we have
        \begin{equation}\label{seconda propr cantor}
           \lim_{R\to+\infty} \frac{dD^c(\psi_{R}\circ u)}{d|D^c(\psi_{R}\circ u)|}=\frac{dD^cu}{d|D^cu|} \quad\text{ $|D^cu|$-a.e. in $ A$}.
        \end{equation}
        \end{itemize}
\end{proposition}
\begin{proof}
      To prove (i), we need to show that for every $B\in\B(A)$ we have  equality \begin{equation}\label{claim i}
          D^c(\phi \circ u)(B)=\nabla \phi(\widetilde{u})D^cu(B).
      \end{equation} 
      We claim that it is enough to prove \eqref{claim i} when $B$ is contained in  $A_m:=\{x\in A\colon \widetilde{u}(x) \text{ exists}$ $  \text{and } |\widetilde{u}(x)|\leq m|\}$, for some $m>0$. Indeed setting $\widetilde{A}:=\{x\in A\colon \widetilde{u}(x) \text{ exists}  \}$, by Proposition \ref{prop: Properties of GBVs scalar} we have that $\hd(\widetilde{A}\setminus J_u)=0$. Since $J_u$ is $\sigma$-finite with respect to $\hd$ and $\phi\circ u\in BV(A;\Rk)$, from these previous observations and (b) of Lemma \ref{lemma:Dcu is a sup} it follows that
      \begin{equation*}
          D^c(\phi \circ u)(A\setminus \widetilde{A})=\nabla\phi(\widetilde{u})D^cu(A\setminus \widetilde{A})=0.
          \end{equation*} Since $\widetilde{A}=\bigcup_{m>0}A_m$, the claim is proved. 

        Let us fix $m>\max_{y\in\textup{supp}(\phi)}|y|$ and $B\in A_m$.  
      To prove \eqref{claim i}, we begin by noting that we have the equality  $\phi\circ u=\phi\circ u^{(m)}$ $\Ld$-a.e. in $A$, so that
    \begin{equation}\nonumber \label{eq:derivata di cantor 1}
    D^c(\phi\circ u)(B)=D^c(\phi\circ u^{(m)})(B).
     \end{equation}
     Since $u^{(m)}\in BV(A;\Rk)$, by the Chain Rule \eqref{eq: Chain Rule for BV}  we have
    \begin{equation}\label{eq:derivata di cantor 2}
        D^c(\phi\circ u^{(m)})(B)=\nabla\phi(\widetilde{u^{(m)}})D^cu^{(m)}(B)=\nabla\phi(\widetilde{u}^{(m)})D^cu^{(m)}(B),
    \end{equation}
    where in the second equality we have used that $\widetilde{u^{(m)}}=\widetilde{u}^{(m)}$. Recalling that $B\subset A_m$, from Lemma \ref{lemma:Dcu is a sup} it follows that   $D^cu(B)=D^cu^{(m)}(B)$. Hence, by  \eqref{eq:derivata di cantor 2} and  $\phi\circ u=\phi\circ u^{(m)}$  $\Ld$-a.e. in $A$ , we get that
    \[
   D^c(\phi\circ u)(B)=D^c(\phi\circ u^{(m)})(B)=\nabla \phi(\widetilde{u}^{(m)})D^cu^{(m)}(B)=\nabla \phi(\widetilde{u})D^cu(B).
    \]
   This concludes the proof of (i).

   We now prove (ii). Let $R>0$. By (i) we have that 
   \begin{equation}
       \nonumber D^c(\psi_R\circ u)=\nabla \psi_R(\widetilde{u})D^{c}u
   \end{equation}
   as Radon measures on $A$. Since $\psi_R$ converges pointwise to the identity as $R\to+\infty$, applying the Dominated Convergence Theorem we deduce both \eqref{eq:limits cantor 1} and \eqref{eq:limits cantor 2}.

  Finally, we prove (iii). Let us fix $R>0$. By (i) and \eqref{eq:PsiR}, it is immediate to see that 
  \begin{equation*}
      \lim_{\rho\to 0^+}\frac{|D^c(\psi_R\circ u)|(B_\rho(x))}{|D^c u|(B_\rho(x))}=1
  \end{equation*}
  for every $x\in A_R$. Using this equality and the Besicovitch Derivation Theorem, we see that
   \begin{equation}\nonumber \label{eq:propo iniziale iii}
    \frac{dD^c(\psi_{R}\circ u)}{d|D^c(\psi_R\circ u)|}(x)=\frac{dD^{c}u}{d|D^cu|}(x) \quad \text{for $|D^cu|$-a.e. $x\in A_R$},
   \end{equation} 
   which proves \eqref{eq:prima propr cantor}. 
   
   To prove \eqref{seconda propr cantor}, we note that since $|D^cu|(A\setminus\bigcup_{R>0}A_R)=0$, passing to the limit for $R\to+\infty$ in \eqref{eq:prima propr cantor} we  obtain \eqref{seconda propr cantor}, concluding the proof.
\end{proof}

A straightforward consequence of this last proposition is that an analogue of the Alberti Rank-One Theorem (see \cite[Corollary 4.6]{alberti_1993} or \cite[Theorem 3.94]{AmbFuscPall}) holds in $GBV_\star(A,\Rk)$.

\begin{corollary}\label{prop:Alberti in GBVs}
    Let $u\in GBV_\star(A;\Rk)$. Then $dD^c u/d|D^cu|$ has rank one $|D^c u|$-a.e. in $A$.
\end{corollary}
\begin{proof}
 If $u\in GBV_\star(A;\Rk)$, by Proposition \ref{prop:characterisation} we have that for every $R>0$ the function $\psi_R\circ u$ belongs to $BV(A;\Rk)$. The Alberti Rank-One Theorem  implies that $dD^c(\psi_R\circ u)/d|D^c(\psi_R\circ u)|$ has rank-one $|D^c(\psi_R\circ u)|$-a.e. in $A$ and thus $|D^cu|$-a.e. in $A_R$. Exploiting \eqref{seconda propr cantor} of Proposition \ref{prop:properties of CantorPart}, the claim follows from the lower semicontinuity of the rank. 
\end{proof}
\section{Lower semicontinuity and Compactness in \texorpdfstring{$GBV_\star(A;\Rk)$}{s}}

In this section we study the lower semicontinuity and coerciveness of some integral functionals defined on $GBV_\star(A;\Rk)$. 
In what follows, $c_1,....,c_4\geq 
0$ are positive constants such that 
\[0<c_1\leq 1\leq c_3.\]
Given $\xi \in\Rkd$, we recall that $|\xi|_{\rm op}$ denotes its operatorial norm.

We will consider functionals whose bulk density  $f:A\times\Rkd\to [0,+\infty)$ satisfies
    \begin{itemize}
         \item [(f1)] $f$ satisfies
        \begin{gather*}
            x\mapsto f(x,\xi) \quad \text{is Borel measurable for every  $\xi\in\Rkd$,}\\
            \xi\mapsto f(x,\xi) \quad \text{is continuous for every $x\in A$},
        \end{gather*}\vspace{-0.3 cm}
        \item [(f2)] $c_1|\xi|_{\rm op}-c_2\leq f(x,\xi)$ \text{ for every }$x\in A$ \text{and every }$\xi\in \Rkd$,
        \item  [(f3)] $f(x,\xi)\leq c_3|\xi|_{\rm op}+c_4$  \text{ for every }$x\in A$ \text{and every }$\xi\in \Rkd$,
    \end{itemize}
    and whose surface integrand $g:A\times\Rk\times\mathbb{S}^{d-1}\to[0,+\infty) $ 
   satisfies
    \begin{itemize}
        \item [(g1)] $g$ is a Borel measurable function,
        \item [(g2)] $c_1(|\zeta|\land 1)\leq g(x,\zeta,\nu)$ \text{ for every }$x\in A$, $\zeta\in \Rk$, $\nu\in\mathbb{S}^{d-1}$,
        \item  [(g3)] $g(x,\zeta,\nu)\leq c_3(|\zeta|\land 1)$  \text{ for every }$x\in A$, $\zeta\in \Rk$, $\nu\in\mathbb{S}^{d-1}$,
        \item [(g4)] $g(x,-\zeta,-\nu)=g(x,\zeta,\nu)$ \text{ for every }$x\in A$, $\zeta\in \Rk$, $\nu\in\mathbb{S}^{d-1}$.
    \end{itemize}
We recall the definition of recession function.
\begin{definition}
    For every $f\colon
    A \times\Rkd\rightarrow[0,+\infty)$, the recession function $f^\infty\colon A \times\Rkd\rightarrow[0,+\infty]$ is  defined as
    \begin{equation}\label{eq:defrecession}\nonumber 
        f^\infty(x,\xi):=\limsup_{t\to +\infty}\frac{f(x,t\xi) }{t}
    \end{equation}
   for every $x\in A$ and for every $\xi\in\Rkd$.
\end{definition}
\begin{remark}
    The function $\xi\to f^\infty(x,\xi)$ is positively $1$-homogeneous. If $f$ satisfies (f1) then $f^\infty$ is Borel measurable and if, in addition, it satisfies (f2) and (f3), then 
    \begin{equation*}\label{eq:Finfty bounds}
         c_1|\xi|_{\rm op}\leq f^\infty(x,\xi)\leq c_3|\xi|_{\rm op}\quad \text{for every  $x\in A $ \text{ and } $\xi\in\Rkd$}.
    \end{equation*}
\end{remark}
\begin{definition}\label{def:Functionals Efg}
 Given $f\colon A\times \Rkd\to [0,+\infty)$ satisfying (f1)-(f3) and $g\colon A\times \Rk\times \mathbb{S}^{d-1}\to [0,+\infty)$ satisfying (g1)-(g4)  we define
 \begin{equation}\label{eq:Def Efg}
     E^{f,g}(u):=\int_A\!f(x,\nabla u)\, dx+\int_Af^\infty\!\Big(x,\frac{dD^cu}{d|D^cu|}\Big)d|D^cu|+\int_{\jump{u}}g(x,[u],\nu_u)\, d\hd
 \end{equation}
  for every $u\in GBV_\star(A;\Rk)$. The definition of $E^{f,g}$ can then be extended to  $L^0(A;\Rk)$ by setting $E^{f,g}(u)=+\infty$ for every $u\in L^0(A;\Rk)\setminus GBV_\star(A;\Rk)$.
\end{definition}

Let $f(x,\xi)=|\xi|_{\rm op}$ and $g(x,\zeta,\nu)=|\zeta|\land 1$. We denote by $V$ the functional $E^{f,g}$ obtained with these choices of $f$ and $g$. Note that in this case $f^\infty(x,\xi)=|\xi|_{\rm op}$, so that 
\[V(u)=\int_A |\nabla u|_{\rm op}\, dx+|D^cu|_{\rm op}(A)+\int_{\jump{u}}|[u]|\land 1\, d\hd\]
for every $u\in GBV(A;\Rk)$.

In \cite[Theorem 2.1]{Braides1995}, the authors prove through a slicing argument a lower semicontinuity result for functionals on the space $\BVsc$. By carefully revisiting their proof, we are able to prove the lower semicontinuity  with respect of the convergence in measure of a subclass of functionals of type $E^{f,g}$, which includes the functional $V$. We recall that a sequence $(u_n)_n\subset BV(A;\Rk)$ is said to converge weakly$^*$ in $BV$ to $u\in BV(A;\Rk)$ if $u_n\to u$ in $L^1(A;\Rk)$ and $Du_n\overset{*}{{\rightharpoonup}}Du$ weakly$^*$ in the sense of $\Rkd$-valued bounded Radon measures.
\begin{lemma}\label{lemma:Lowersemicontinuity weakly in BV}
    Let $\hat{f}\colon[0,+\infty)\rightarrow[0,+\infty)$ be a convex, non-decreasing, lower semicontinuous  function such that 
    \[\hat{f}^\infty(1)=\lim_{t\to+\infty}\frac{\hat{f}(t)}{t}=C.
    \] Let $\hat{g}\colon[0,+\infty)\rightarrow[0,+\infty)$ be a lower semicontinuous function satisfying the subadditivity condition
    \[\hat{g}(a+b)\leq \hat{g}(a)+\hat{g}(b) \quad \text{for every $a,b\in\R$},\]
    and suppose in addition that 
    \[\hat{g}^0(1)=\lim_{t\to 0^+}\frac{\hat{g}(t)}{t}=C.\]
Set $f:=\hat{f}(|\cdot|_{\rm op})$ and $g:=\hat{g}(|\cdot|)$.
    Then the functional defined for $u\in \BVvect$ as 
    \begin{equation}\nonumber  E^{f,g}(u):=\int_A \hat{f}(|\nabla u |_{\rm op})\, dx+C|D^cu|_{\rm op}(A)+\int_{\jump{u}}\hat{g}(|[u]|)\,d\hd \end{equation}
    is weakly$^*$ lower semicontinuous in $BV(A;\Rk)$.
\end{lemma}
\begin{proof}The proof follows closely that of \cite[Theorem 2.1]{Braides1995}. For the reader's convenience, we give a sketch of the proof.

  For every $\nu\in\mathbb{S}^{d-1}$, we introduce the functional $E^{f,g}_\nu\colon \BVvect\times\mathcal{B}(A)\to [0,+\infty)$ defined for every $B\in \mathcal{B}(A)$ and $u\in\BVvect$ by 
  \[ E^{f,g}_\nu(u,B)=\int_B \hat{f}(|(\nabla u)\nu |)\, dx+C|(D^cu) \nu|(B)+\int_{\jump{u}\cap B}\hat{g}(|[u]|)|\nu_u\cdot \nu |d\hd.\]

  We want to prove that for every $U\in\mathcal{A}(A)$, the functional
  $E^{f,g}_\nu(\cdot,U)$ is lower semicontinuous with respect to the weak$^*$ convergence in $BV(A;\Rk)$.
  
  Since for every $u\in BV(A;\Rk)$ the set function $E^{f,g}_\nu(u,\cdot)$ is a non-negative Borel measure such that $E^{f,g}_\nu(u,B)=0$ whenever $\Ld(B)+|D^cu|(B)$ $+|D^ju|(B)=0$, it is enough to check that $E^{f,g}_\nu(\cdot,Q)$  is lower semicontinuous when $Q=\Sigma\times I$ for an open $\Sigma\subset \Pi_\nu=\{x\cdot \nu=0\}$   and $I\in\A(\R)$.
  
  Thanks to Proposition \ref{prop:Slicing} and arguing exactly as in \cite[Lemma 2.3]{Braides1995}, we deduce that, for every open $\Sigma\subset \Pi_\nu$ and  $I\in\A(\R)$, setting $Q=\Sigma\times I$, we have 
  \begin{equation}\label{eq:identity functional Fv and integral of 1d slice}
  E^{f,g}_\nu(u,Q)=\int_\Sigma\Phi(u_y^\nu,I)\,d\hd(y)\,\,\,\text{ for every $u\in BV(A;\Rk)$,}
   \end{equation}
  where $\Phi\colon BV(\R;\Rk)\times\mathcal{B}(\R)\to [0,+\infty]$ is the functional defined for every $I\in\mathcal{B}(\R)$ and
$v\in BV(\R;\Rk)$ as 
  \[\Phi(v,I)=\int_I \hat{f}(|\nabla v|_{\rm op})\, dx+C|D^cv|(I)+\int_{\jump{v}\cap I}\hat{g}(|[v]|)\,d\mathcal{H}^0 \]
  and where for every $y\in\Pi_\nu$ the function  $u^\nu_y$ is  defined by \eqref{eq:slice}.
  
 Let $u\in \BVvect$ and  $(u_n)_n\subset\BVvect$ be such that $u_n\rightharpoonup u $ weakly$^*$ in $\BVvect$. Thanks to Proposition \ref{prop:Slicing}(c), we have that for $\hd$-a.e. $y\in \Sigma$
  \[(u_n)_y^\nu\rightharpoonup u_y^\nu\text{ weakly$^*$ in $BV(I;\Rk)$. }\]
  Thanks to \cite[Theorem 3.3]{Buttazzo-Bouchitte}, we know that  $\Phi$ is weakly$^*$ lower semicontinuous on $BV(\R;\Rk)$, so that Fatou's Lemma and \eqref{eq:identity functional Fv and integral of 1d slice} imply that
\begin{equation*}\label{eq:l.s.c of Fnu on cilinders}
    E^{f,g}_\nu(u,Q)\leq\liminf_{n\to\infty}E^{f,g}_\nu(u_n,Q),
\end{equation*}

 Fix now $(\nu_i)_{i\in\N}$ a dense family of unit vectors in $\mathbb{S}^{d-1}$. To prove that $E^{f,g}$ is weakly$^*$ lower semicontinuous in $\BVvect$, we show that it can be written as the supremum of finite sums of functionals of the form $E^{f,g}_{\nu_i}$, which we have just shown to be lower semicontinuous. More precisely, repeating the proof of \cite[Lemma 2.4]{Braides1995} replacing 
 their measure $\mu$ by
 \begin{equation}
     \nonumber
     \mu=\Ld+|D^cu|_{\rm op} +\hd\mres J_u, 
 \end{equation}
 their function $\psi$ by 
 \begin{equation*}\psi(x)=
     \begin{cases}
         \hat{f}(|\nabla u(x)|_{\rm op}) &\Ld\text{-}a.e\,\, \text{in }A, \\
         C &|D^cu|\text{-}a.e.\text{ in }A,\\
         \hat{g}(|[u]|)  & \text{ $\hd$-a.e. in }\jump{u},
     \end{cases}\\
     \end{equation*}
    and their function  $\psi_h$ by 
    \begin{equation*}
     \psi_i(x)=\begin{cases}
          \hat{f}(|(\nabla u) \nu_i|) &\Ld\text{-}a.e\,\, \text{in }A, \\
          C|(\frac{dD^cu}{d|D^cu|}) \nu_i|&|D^cu|\text{-}a.e.\text{ in } A,\\
         \hat{g}(|[u]|)|\nu_u\cdot\nu_i|  & \hd \text{-}a.e.\text{ in }\jump{u},
     \end{cases}
\end{equation*}
we get 
\begin{eqnarray*} E^{f,g}(u)=\sup\Big\{\sum_{i=1}^nE^{f,g}_{\nu_i}(u,A_i):&& \!\!\!\!\!\!\!\!\!\!n\in\mathbb{N} \text{ and}\\&&\hspace{-1.7 cm}(A_i)_{i=1}^n\,\,\ \text{pairwise disjoint open subsets of $A$} \Big\},
\end{eqnarray*}
concluding the proof.
\end{proof}

\begin{theorem}\label{thm:lowersemicontinuity}
    Let $\hat{f},\hat{g}$ be as in the statement of Lemma \ref{lemma:Lowersemicontinuity weakly in BV}. Assume  that $f:=\hat{f}(|\cdot|_{\rm op})$ satisfies \text{\rm (f1)-(f3)} and  that $g:=\hat{g}(|\cdot|)$ satisfies \text{\rm (g1)-(g3)}
    and that $\hat{g}$ is non-decreasing.
    Then the functional $E^{f,g}$ of Definition \ref{def:Functionals Efg} is lower semicontinuous with respect to the topology of $L^0(A;\Rk)$.
\end{theorem}

\begin{proof}

   We consider $u\in L^0(A;\Rk)$ and a sequence  $(u_n)_n\subset L^0(A;\Rk)$ converging to $u$ in $L^0(A;\Rk)$. If $\liminf_n E^{f,g}(u_n,A)=+\infty$ there is nothing to prove. Thus, we suppose that $\sup_{n\in\N}E^{f,g}(u_n)\leq M$ for some $M>0$. Since $f$ and $g$ satisfy (f2) and (g2), we have  that $\sup_{n\in\N}V(u_n)\leq M'$, for $M'$ a constant depending only on $M.$ With arguments similar to those used in the proof of Proposition \ref{prop:characterisation}, one can show that the previous condition implies that $u\in GBV_\star(A;\Rk)$ as well.

    For every $R>0$ and every $n\in\N$ we set $v_n^R:=\psi_R\circ u_n$, where $\psi_R$ is the function defined by \eqref{eq:defpsi}. By Proposition \ref{prop:characterisation} we have that $v_n^R\in BV(A;\Rk)\cap L^\infty(A;\Rk)$. Arguing as in the proof of Proposition \ref{prop:characterisation} we see that there exists $K=K(R,M')>0$ such that
    \begin{equation}\label{eq:vRn is uniformly bounded}
    |Dv_{n}^R|(A)\leq  K.   
    \end{equation}
   
   The sequence $v_{n}^R$ converges to $\psi_R\circ u$ in $L^1(A;\Rk)$ as $n\to+\infty$, so that from \eqref{eq:vRn is uniformly bounded} it follows that $v_{n}^R$ converges to $\psi_R\circ u$ weakly$^*$ in $\BVvect$ as well. We can then use
   Lemma \ref{lemma:Lowersemicontinuity weakly in BV} to obtain
    \begin{equation}\label{eq:lowersemicontinuity befor passing to the limit in m}
   E^{f,g}(\psi_R\circ u)\leq\liminf_{n\to+\infty} E^{f,g}(v_{n}^R)
   \end{equation}
   for every $R>0$.
   By the chain rule \eqref{eq: Chain Rule for BV}, we may estimate
    \begin{align*}
        \displaystyle E^{f,g}(v_{n}^R)&=\int_A \hat{f}(|\nabla \psi_R(u_n)\nabla u_n|_{\rm op})\, dx+C|\nabla \psi_R(\widetilde{u}_n)D^cu_n|_{\rm op}(A) \\
        \displaystyle &\hspace{0.5 cm}+ \int_{\jump{u}}\hat{g}(|[\psi_R(u_n)]|)\, d\hd\\
        &\leq \int_A \hat{f}(|\nabla u_n|_{\rm op})\, dx+C|D^cu_n|_{\rm op}(A) 
        \displaystyle +\int_{\jump{u}}\hat{g}(|[u_n]|)\, d\hd,
    \end{align*}
    where in the last inequality we have used  that $\hat{f}$ and $\hat{g}$ are non-decreasing, \eqref{eq:PsiRLip}, and Proposition \ref{prop:psi sui salti}. From this last estimate and \eqref{eq:lowersemicontinuity befor passing to the limit in m}, we get 
    \begin{equation}\label{eq: Per concludere lower}
        E^{f,g}(\psi_R\circ u)\leq\liminf_{n\to+\infty}E^{f,g}(u_n).
    \end{equation}
   
    Recalling Proposition \ref{prop: Properties of GBVs scalar}(b), (d) and taking advantage of the lower semicontinuity of $\hat{f}$ and of $\hat{g}$, we get
    \begin{eqnarray}
      & \label{eq:lower f}\displaystyle \hat{f}(|\nabla u|_{\rm op})\leq\liminf_{R\to+\infty} \hat{f}(|\nabla (\psi_R\circ u)|_{\rm op})\quad  &\Ld\text{-}a.e\,\, \text{ in } A,\\
      &\displaystyle \label{eq:lower g}\hat{g}(|[u]|)\leq\liminf_{R\to+\infty} \hat{g}(|[\psi_R\circ u]|)\quad  &\hd\text{-} a.e\,\, \text{ in } \jump{u}.
    \end{eqnarray}

   Finally, \eqref{eq:lower f}, \eqref{eq:lower g}, and \eqref{eq:limits cantor 2}, together with Fatou's Lemma and \eqref{eq: Per concludere lower}, imply
    \[E^{f,g}(u)\leq\liminf_{n\to+\infty} E^{f,g}(u_n),\]
   concluding the proof.
\end{proof}
\begin{corollary}
  The functional $V$ is lower semicontinuous with respect to the topology of $L^0(A;\Rk)$. 
\end{corollary}

In \cite[Theorem 3.11]{DalToa22} a useful compactness theorem for $GBV_\star(A)$ is proved. We now show that the result readily adapts to $GBV_\star(A;\Rk)$.
\begin{theorem}\label{thm:Compactness naive}
Let $(u_n)_n$ be a sequence in $GBV_\star(A;\Rk)$. Suppose that there exist a constant $M>0$ and an increasing continuous function $h\colon[0,+\infty)\rightarrow[0,+\infty)$ with $h(t)\to+\infty$ as $t\to+\infty$, such that
\begin{eqnarray}& \notag \displaystyle
\sup_n V(u_n)\leq M,
\\&  \displaystyle\label{eq:increasingpsicompactness}\sup_n\int_A h(|u_n|)\, dx<+\infty.
\end{eqnarray}
Then there exists a subsequence, not relabelled, and $u\in GBV_\star(A;\Rk)$ such that $u_n\to u$ $\mathcal{L}^d$-a.e. in $A$. 
\end{theorem}
\begin{proof}
    It is enough to apply Theorem \cite[Theorem 3.11]{DalToa22} to each component of $u$.
\end{proof}

 We now present a result which shows that given $(u_n)_n\subset\GBVsvector$  which is only bounded in energy and which satisfies some common Dirichlet boundary condition, it is possible to produce a modification $y_n$ of $u_n$, satisfying the hypotheses of Theorem \ref{thm:Compactness naive}. This result is a direct adaptation of \cite[Theorem 5.5]{DalToa22} and of \cite[Theorem 7.13]{DalToa23}, whose proof is based on the  arguments of \cite{friedrich2019compactness}.

In the following  $A'\subset\subset A$ is an open set with Lipschitz boundary. We also assume $A$ to have Lipschitz boundary.
 We fix an additional positive constant $c_5>0$ satisfying $c_5\geq c_3/c_1$ and make  the following additional assumption on the integrand $g$:
\begin{itemize}
    \item [(g5)]for every $\zeta_1,\zeta_2\in\Rk$ with $c_5|\zeta_1|\leq |\zeta_2|$ it holds $g(x,\zeta_1,\nu)\leq g(x,\zeta_2,\nu)$ for every $x\in\Rd$, $\nu\in\mathbb{S}^{d-1}$.
\end{itemize}
Note that this condition, first considered in \cite{Cagnetti}, is crucial in the proof of the result.


\begin{theorem}\label{thm:Compactness}
   Let  $f\colon A\times \Rkd\to [0,+\infty)$  and  $g\colon A\times \Rk\times \mathbb{S}^{d-1}\to [0,+\infty)$ be two functions satisfying {\rm (f1)-(f3)} and {\rm (g1)-(g5)}, respectively, and let $E^{f,g}$ be the functional introduced in Definition \ref{def:Functionals Efg}. Let $w\in W^{1,1}(A;\Rk)$ and let $ (u_n)_n\subset GBV_\star(A;\Rk)$ with $u_n=w$ $\Ld$-a.e. on $A\setmeno\overline{A'}$, and $V(u_n,A)\leq M$ for every $n\in\mathbb{N}$. Then for every $\e_n\to 0^+$ there exists a subsequence of $(u_n)_n$, not relabelled, modifications $y_n\in GBV_\star(A;\Rk)$ of $u_n$, with $y_n=w$ on $A\setmeno\overline{A'}$, and a continuous increasing function $h:[0,+\infty)\rightarrow[0,+\infty)$ satisfying $h(t)\to +\infty$ for $t\to +\infty$, such that 
    \begin{eqnarray}
        &\label{eq: yk is minimizing}\displaystyle E^{f,g}(y_n)\leq E^{f,g}(u_n)+\e_n,\\
        &\label{eq:Existence of psi}\displaystyle \sup_n\int_{A}h(|y_n|)\, dx<+\infty.
    \end{eqnarray}
\begin{proof}
The theorem follows from the same lines of proof of \cite[Theorem 5.5]{DalToa22} outlined in \cite[Section 5]{DalToa22}, performing the modifications suggested in \cite[Theorem 7.13]{DalToa23} and replacing \cite[Lemma 5.1]{DalToa22} by our Lemma \ref{lemma:Linftyapprox} below.
\end{proof}
\end{theorem}

Given an $\Ld$-measurable set $E$, we recall that  a countable collection $(P^j)_{j}$ of $\Ld$-measurable subsets of $E$ is said to be a Caccioppoli partition of $E$ if 
\begin{eqnarray*}
    &\displaystyle\Ld\Big(E\setminus \bigcup_{j=1}^\infty P^j\Big)=0,\\
    &\displaystyle \sum_{j=1}^\infty\hd(\partial^*P^j\cap E)<+\infty.
\end{eqnarray*}
\begin{lemma}\label{lemma:Linftyapprox}
    For every $M>0$ and $u\in \GBVsvector$ satisfying 
    \[ V(u)\leq M,\]
    there exists a Caccioppoli partition $(P^j)_j$ of $A$ and a family of translations 
    $(t^j)_{j\in\N}\subset\Rk$ such that the function
    \begin{equation*}\label{eq:deftranslated}
    v=u-\sum_{j=1}^\infty t^j\chi_{P^j}
      \end{equation*}
    is in $\BVvect\cap L^\infty(A,\Rk)$ and the following estimates hold
\begin{eqnarray}
  \label{eq:caccioppolivector}   & \displaystyle \sum_{j=1}^\infty\mathcal{H}^{d-1}(\partial^*P^j)\leq k(2+2M+\mathcal{H}^{d-1}(\partial A)),\\
   & \label{eq:caccioppolinormvector} \displaystyle\|v\|_{L^\infty(A;\Rk)}\leq 2M.
\end{eqnarray}
\end{lemma}
\begin{proof}
Thanks to \cite[Lemma 5.1]{DalToa22}, for every $i=1,...,k$ we can find a Caccioppoli Partition $(P^j_i)_{j}$ and a family of translations $(t_i^j)_{j\in\N}\subset\R$ satisfying 
\begin{eqnarray}
   \label{eq:caccioppolicomponente}  & \displaystyle\sum_{j=1}^\infty\mathcal{H}^{d-1}(\partial^*P_i^j)\leq 2+2M+\mathcal{H}^{d-1}(\partial A),\\
   & \displaystyle\label{eq:caccioppolinormcomponente}\|u_i-\sum_{j=1}^\infty t^j_i\chi_{P_i^j}\|_{L^\infty(A)}\leq 2M.
\end{eqnarray}
The family
\begin{eqnarray}\nonumber 
    \mathcal{P}:=\{P\subset A\!: P=\cap_{i=1}^kP^{j_i}_i \text{ for $(j_1,...,j_k)\in\mathbb{N}^k$}\}
\end{eqnarray}
is countable and is still a partition of $A$. We denote every element of $\mathcal{P}$ as $P_\ell$ for some $\ell\in\mathbb{N}$. Thanks to \eqref{eq:caccioppolicomponente} and to the standard inequality 
\[\mathcal{H}^{d-1}(\partial^*(E_1\cap E_2)\cap A)\leq \mathcal{H}^{d-1}(\partial^*E_1\cap A)+\mathcal{H}^{d-1}(\partial^*E_2\cap A), \]
 which holds for every $E_1$ and 
$E_2$ of finite perimter, we get that $(P^\ell)_{\ell}$ satisfies \eqref{eq:caccioppolivector}. In particular, $(P^\ell)_\ell$ is a Caccioppoli partition of $A$.

For every $\ell\in\N$ let  $(j^\ell_1,...,j_k^\ell)$ be such that $P_\ell=\cap_{i=1}^kP_i^{j^\ell_i}$ and let $(t_1^{j^\ell_1},...,t_k^{j^\ell_k})$ be as above. We set $t^\ell=(t_1^{j_1^\ell},...,t_k^{j_k^\ell})$ and 
\begin{equation*}
    v:=u-\sum_{\ell=1}^{+\infty} t^\ell\chi_{P^\ell}.
\end{equation*}
Recalling \eqref{eq:caccioppolinormcomponente}, we infer that $v$ satisfies \eqref{eq:caccioppolinormvector}, so that Lemma \ref{lemma:Linftyapprox} is finally proved.
\end{proof}

\begin{remark}

Let $f\colon A\times \Rkd\to [0,+\infty) $ satisfying (f1)-(f3), let $ g\colon A\times \Rk\times \mathbb{S}^{d-1}\to [0,+\infty)$ satisfying (g1)-(g5), and let  $w\in W^{1,1}(A;\Rk)$. Consider  $(u_n)_n\subset GBV_\star(A;\Rk)$ a minimising sequence for the problem
\begin{equation}\label{eq:minimum problem}
\inf\{E^{f,g}(u)\colon u\in GBV_\star(A;\Rk), \,\,u=w\,\,\Ld\text{-} a.e. \text{ on } A\setminus\overline{A'}\},
\end{equation}
where $E^{f,g}$ is as in Definition \ref{def:Functionals Efg}.

Thanks to (f2)-(f3), (g2)-(g3), and \eqref{eq:Finfty bounds}, it is easily seen that for some $M>0$ we have $V(u_n)\leq M$  for every $n\in\N$. Hence, by Theorem \ref{thm:Compactness} there exist a minimising sequence $(y_n)_n$, with $y_n=w$ on $A\setminus\overline{A'}$ and satisfying \eqref{eq:Existence of psi}. In particular, $(y_n)_n$ satisfies the hypotheses of Theorem \ref{thm:Compactness naive}, so that there exists a subsequence of $(y_n)_n$, not relabelled, converging in $L^0(A;\Rk)$ to a function $y$, with $y=w$ $\Ld$-a.e. in $A\setminus\overline{A'}$. If the functional $E^{f,g}$ is lower semicontinuous with respect to the convergence of $L^0(A;\Rk)$ (for instance if it satisfies the hypoteses of Theorem
\ref{thm:lowersemicontinuity}), we conclude that $y$ minimises \eqref{eq:minimum problem}.
\end{remark}
\bigskip
\noindent \textsc{Acknowledgements.}
The author wishes to thank Professor Gianni Dal Maso for all the useful  discussions on the topic.
This paper is based on work supported by the National Research Project (PRIN  2022J4FYNJ) \lq\lq Variational methods for stationary and evolution problems with singularities and interfaces\rq\rq, funded by the Italian Ministry of University and Research. 
The author is a member of the Gruppo Nazionale per 
l'Analisi Matematica, la Probabilit\`a e le loro Applicazioni (GNAMPA) of the 
Istituto Nazionale di Alta Matematica (INdAM).

\emergencystretch=1em

\printbibliography

\end{document}